\documentclass[preprint,12pt]{elsarticle}

\usepackage{color}

\usepackage{proof}
\usepackage{amssymb}
\usepackage{amsmath}
\usepackage{amsfonts}
\usepackage{amsthm}
\usepackage{float}
\usepackage{times}

\floatstyle{boxed}
\newtheorem{Theorem}{Theorem}
\newtheorem{Proposition}{Proposition}
\newtheorem{Lemma}{Lemma}
\newtheorem{Corollary}{Corollary}

\theoremstyle{definition}
\newtheorem{Example}{Example}

\newenvironment{newlist}
   {\begin{list}{}{\setlength{\labelsep}{0.15cm}
\setlength{\itemsep}{0cm}
\setlength{\topsep}{0.15cm}
                   \setlength{\labelwidth}{0.7cm}
                      \setlength{\leftmargin}{0.75cm}}}  
   {\end{list}}
   
\newcommand{\setc}{\mathnormal{\setminus}}

\newcommand{\de}{\delta}

\newcommand{\e}{e}

\newcommand{\m}{\mathbf} 
\newcommand{\R}{\mathbb{R}}
\newcommand{\Z}{\mathbb{Z}}
\newcommand{\N}{\mathbb{N}}
\newcommand{\iv}[1]{{#1}^{-1}}

\newcommand{\Si}{\mathrm{\Sigma}}
\newcommand{\Om}{\mathrm{\Omega}}

\newcommand{\vty}[1]{\mathcal{#1}}

\newcommand{\eq}{\approx}

\newcommand{\fr}[1]{#1}
\newcommand{\qr}[1]{\overline{#1}}
\newcommand{\subg}[1]{\mathsf{B}({#1})}
\newcommand{\subgen}[2]{\mathsf{B}_{#2}({#1})}

\newcommand{\subrg}[1]{\mathsf{R}(#1)}
\newcommand{\subrgen}[2]{\mathsf{R}_{#2}({#1})}

\newcommand{\sgr}[1]{\langle #1 \rangle} 
\newcommand{\nsgr}[1]{{\langle\!\langle #1 \rangle\!\rangle}}

\begin{document}

\begin{frontmatter}

\title{Ordering groups and validity in lattice-ordered groups}

\author{Almudena Colacito\fnref{thanks}}
\address{Mathematical Institute, University of Bern, Sidlerstrasse 5, 3012 Bern, Switzerland}
\ead{almudena.colacito@math.unibe.ch}

\author{George Metcalfe\corref{cor}\fnref{thanks}}
\address{Mathematical Institute, University of Bern, Sidlerstrasse 5, 3012 Bern, Switzerland}
\ead{george.metcalfe@math.unibe.ch}
\cortext[cor]{Corresponding author}
\fntext[thanks]{Supported by Swiss National Science Foundation grant 200021{\_}165850 and the EU Horizon 2020 research and innovation programme under the Marie Sk{\l}odowska-Curie grant agreement No 689176.}	

\begin{abstract}
A characterization is given of the subsets of a group that extend to the positive cone of a right order on the group and used to relate validity of equations in lattice-ordered groups ($\ell$-groups) to subsets of free groups that extend to positive cones of right orders. This correspondence is used to obtain new proofs of the decidability of the word problem for free $\ell$-groups and generation of the variety of $\ell$-groups by the $\ell$-group of automorphisms of the real number line. A characterization of the subsets of a group that extend to the positive cone of an order on the group is also used to establish a correspondence between the validity of equations in varieties of representable $\ell$-groups (equivalently, classes of ordered groups) and subsets of relatively free groups that extend to positive cones of orders.
 \end{abstract}


\end{frontmatter}
 

\section{Introduction}\label{s:introduction}

The first aim of this paper is to establish a correspondence between validity of equations in lattice-ordered groups ($\ell$-groups) and subsets of free groups that extend to positive cones of right orders on the group, thereby relating validity in $\ell$-groups also to properties of the topological spaces of right orders on free groups. This correspondence is used to obtain new proofs of the decidability of the word problem for free $\ell$-groups and generation of the variety of $\ell$-groups by the $\ell$-group of automorphisms of the real number line. A correspondence is also established between validity of equations in varieties of representable $\ell$-groups (equivalently, classes of ordered groups) and subsets of relatively free groups that extend to positive cones of orders on the group. Our main tools for proving these results will be ordering theorems for groups that stem from proof-theoretic investigations into $\ell$-groups, and require very little structural theory for these algebras.

Recall first, referring to~\cite{KM94} for further details and references, that an $\ell$-group is an algebraic structure $\m L = \langle L, \land, \lor, \cdot, \iv{}, \e \rangle$  such that  $\langle L,\cdot, \iv{}, \e \rangle$ is a group and $\langle L, \land, \lor \rangle$ is a lattice with an order $a \le b \,:\Leftrightarrow\,a \land b = a$ that is compatible with left and right multiplication. If $\le$ is also total, then $\m L$ is  an ordered group (o-group). The class $\vty{LG}$ of all $\ell$-groups forms a variety, and the class of o-groups generates the variety $\vty{RG}$ of representable $\ell$-groups.  By a fundamental theorem of Holland, every $\ell$-group embeds into the group $\m{Aut}(\langle \Om, \le \rangle)$ of order-preserving bijections of a totally-ordered set $\langle \Om,\le\rangle$ equipped with coordinate-wise lattice operations~\cite{Hol63}. This representation was used by Holland to prove that $\vty{LG}$  is generated as a variety by $\m{Aut}(\langle \R,\le \rangle)$~\cite{Hol76}  and, with McCleary, to establish the decidability of the word problem for free $\ell$-groups~\cite{HM79}. Alternative proofs of these theorems that avoid the use of Holland's embedding theorem and sharpen the decidability result to co-NP completeness were given by Galatos and Metcalfe~\cite{GM16}; this approach forms part of a broader program that aims to develop relationships between $\ell$-groups and varieties of residuated lattices (see also~\cite{BKLT16,GLPT16,GLT17}).  

Recall next that a partial order $\le$ on a group $\m{G}$ with neutral element $\e$ is  a partial right order of $\m{G}$ if it is compatible with right multiplication in $\m{G}$, and a  right order if it is also total. The positive cone $P_\le = \{a \in G \mid \e < a\}$ of a partial right order of $\m{G}$ is always a subsemigroup of $\m{G}$ (i.e., $a,b \in P_\le$ implies $ab \in P_\le$) that omits $\e$. Conversely, if $P \subseteq G$ is a subsemigroup of $\m{G}$ omitting $\e$, then $a \le^P b\,:\Leftrightarrow\, b\iv{a} \in P \cup \{\e\}$ defines a partial right order of $\m{G}$ satisfying $P_{\le^P} = P$. Hence from now on, we will identify partial right orders of $\m{G}$ with subsemigroups of $\m{G}$ that omit $\e$, and right orders with partial right orders $P$ such that $a \in P$ or $\iv{a} \in P$ for all $a \in G \setc \{\e\}$. In particular, each right order of $\m{G}$ can be viewed as a subset of $G$ and the set $\mathcal{RO}(\m{G})$ of right orders of $\m{G}$ endowed with the subspace topology inherited from the powerset $2^G$ with the product topology forms a compact totally disconnected topological space. This topological perspective on right orders is explored in detail in~\cite{CR16,deroin2014groups}. 

We prove here that a finite subset $\{\fr{t}_1,\ldots,\fr{t}_n\}$ of a free group $\m{F}$ extends to a right order if and only if the inequation $\e \le t_1 \lor \dots \lor t_n$ fails in some $\ell$-group (Theorem~\ref{t:secondmainright}). Since every $\ell$-group term is equivalent in $\vty{LG}$ to a term of the form $\bigwedge_{i \in I} \bigvee_{j \in J_i} t_{ij_i}$ where each  $t_{ij_i}$ is a group term, this correspondence provides a full characterization of validity in  $\vty{LG}$. The result may be established using Hollister's theorem~\cite{holl65} that the lattice order of an $\ell$-group is the intersection of right orders on its group reduct. However, we  use here instead an inductive characterization of subsets of groups that extend to right orders (Theorem~\ref{t:mainrightorderthm}, closely related to a theorem of Conrad~\cite{Con59}), to obtain a proof that requires almost no structure theory of $\ell$-groups. We then make use of the correspondence to obtain new proofs of the generation of $\vty{LG}$ as a variety by $\m{Aut}(\langle \R,\le \rangle)$ and the decidability of the word problem for free $\ell$-groups, the latter by appealing to an algorithm by Clay and Smith  that recognizes when a given finite subset of a finitely generated free group extends to a right order~\cite{CS09}. 

In the last part of the paper we turn our attention to validity of equations in varieties of representable $\ell$-groups (equivalently, classes of o-groups). A partial order $\le$ on $G$ that is compatible with left and right multiplication in $\m{G}$ is called a partial order of $\m{G}$, and an order if it is total. Positive cones of partial orders of  $\m{G}$ then correspond to normal (i.e., closed under conjugation) subsemigroups of $\m{G}$ omitting $\e$. We provide an inductive characterization of subsets of groups that extend to orders (Theorem~\ref{t:mainorderthm}, closely related  to a theorem of Ohnishi~\cite{OM52}), and use this to establish that a finite subset $\{\fr{t}_1,\ldots,\fr{t}_n\}$ of a relatively free group of a variety $\vty{V}$ of groups extends to an order if and only if the inequation $\e \le t_1 \lor \dots \lor \,t_n$ fails in some o-group with group reduct in $\vty{V}$ (Theorem~\ref{t:secondmain}). 

Let us remark finally that a proof-theoretic account of some of the results presented here is given in the conference paper~\cite{CM17}. Note also that Wessel in~\cite{Wes17} develops an alternative syntactic approach to orderability of groups that yields a new proof of a finitary version of Sikora's theorem.


\section{A right ordering condition for groups}\label{s:2}

Let us fix a group $\m{G}$ with neutral element $\e$ and denote by $\sgr{S}$ the subsemigroup of $\m{G}$ generated by $S \subseteq G$. Clearly, $\sgr{S}$ is a partial right order of $\m{G}$ if and only if $\e \not \in\sgr{S}$. The following characterization of the subsets of $G$ that extend to right orders of $\m{G}$ is proved by a straightforward application of Zorn's lemma (see~\cite{BR77}):

\begin{newlist}
\item[$(\dagger)$]
 $S\subseteq G$ extends to a right order of $\m{G}$ if and only if for all $a_1,\ldots,a_n \in G\setc\{\e\}$, there exist $\de_1,\ldots,\de_n \in \{-1,1\}$ satisfying $\e \not\in \sgr{S \cup \{a_1^{\de_1},\ldots,a_n^{\de_n}\}}$.
\end{newlist}

\noindent
We make use here, however, of an alternative inductive description of these sets (similar to characterizations by Ohnishi~\cite{OM52} for orderable groups, and Conrad~\cite{Con59} for right-orderable groups) that is more suitable for establishing relationships with validity in $\ell$-groups. We define inductively for $n \in \N$,
\[
\begin{array}{rcl}
\subrgen{\m{G}}{0} & = & \{S \subseteq G \mid S \cap S^{-1} \neq \emptyset\};\\[.1in]
\subrgen{\m{G}}{n+1} & = & \subrgen{\m{G}}{n} \cup \{T \cup \{ab\} \mid T \cup \{a\}, T \cup \{b\} \in \subrgen{\m{G}}{n}\};\\[.1in]
\subrg{\m{G}} & =  &\bigcup_{n \in \N} \subrgen{\m{G}}{n}.
\end{array}
\]
It follows that $\subrg{\m{G}} \subseteq \mathcal{P}(G)$ is the smallest set containing $\subrgen{\m{G}}{0}$ satisfying the condition that whenever $T \cup \{a\}, T \cup \{b\}  \in \subrg{\m{G}}$, also $T \cup \{ab\} \in \subrg{\m{G}}$. Since any $S \subseteq G$ occurring in $\subrg{\m{G}}$ must occur in $\subrgen{\m{G}}{n}$ for some $n \in \N$, there exists in this case a finite tree of subsets of $G$ with root $S$ and leaves in $\subrgen{\m{G}}{0}$ such that each non-leaf node is of the form $T \cup \{ab\}$ and is obtained from its parent nodes $T \cup \{a\}, T \cup \{b\}$ using the above condition. 

\begin{Example}
Let $\m{Z}$ be the additive group of the integers. Then $\{1,-1\} \in \subrgen{\m{Z}}{0}$, so $\{1,-2\} \in  \subrgen{\m{Z}}{1}$. But $\{2,-2\} \in  \subrgen{\m{Z}}{1}$, so $\{3,-2\} \in  \subrgen{\m{Z}}{2}$, and since also $\{3,-3\} \in  \subrgen{\m{Z}}{2}$, it follows that  $\{3,-5\} \in \subrgen{\m{Z}}{3} \subseteq \subrg{\m{Z}}$. This chain of reasoning can be displayed as a tree of finite sets of integers as follows:
\[
\infer{\{3,-5\}}{
 \infer{\{3,-3\}}{} &
 {\infer{\{3,-2\}}{
  {\infer{\{2,-2\}}{}} &
{\infer{\{1,-2\}}{
    {\infer{\{1,-1\}}{}} &
     {\infer{\{1,-1\}}{}}}}}}}
\]
It is easily proved that  $S \subseteq \Z$ is in $\subrg{\m{Z}}$ if and only if $S$ contains elements $m \le 0$ and $n \ge 0$. This corresponds to the fact that $S \subseteq \Z$ extends to a (right) order of $\m{Z}$ (of which there are just two, the standard one and its dual) if and only if it contains only strictly positive elements or only strictly negative elements (see Theorem~\ref{t:mainrightorderthm}).
\end{Example}

\begin{Example}
Let $\m{F}(2)$ denote the free group on two generators $x$ and $y$. The following tree of subsets of $F(2)$ demonstrates that $\{xx, yy, \iv{x}\iv{y}\} \in \subrg{\m{F}(2)}$:
\[
 \infer{\{xx, yy, \iv{x}\iv{y}\}}{
  \infer{\{xx, yy, \iv{x}\}}{
   \infer{\{x, yy, \iv{x}\}}{}
   &
   \infer{\{x, yy, \iv{x}\}}{}}
   &
    \infer{\{xx, yy, \iv{y}\}}{
  \infer{\{xx, y, \iv{y}\}}{}
  &
\infer{\{xx, y, \iv{y}\}}{}}}
\]
This corresponds to the fact that $\{xx, yy, \iv{x}\iv{y}\}$ does not extend to a right order of $\m{F}(2)$ (see Theorem~\ref{t:mainrightorderthm}) and also the fact that the inequation $\e \le xx \lor yy \lor \iv{x}\iv{y}$ is valid in all $\ell$-groups (see Theorem~\ref{t:secondmainright}).
\end{Example}

The remainder of this section is devoted to proving the following result.

\begin{Theorem}\label{t:mainrightorderthm} 
\
\begin{newlist}
\item[\rm (a)]
A group $\m{G}$ is right-orderable if and only if $\{a\} \not\in \subrg{\m{G}}$ for all $a \in G\setc\{\e\}$.
\item[\rm (b)]
If a group $\m{G}$ is right-orderable, then $S \subseteq G$ extends to a right order of $\m{G}$ if and only if $S \not \in \subrg{\m{G}}$.
\end{newlist}
\end{Theorem}

\noindent
We first establish some elementary properties of $\subrg{\m{G}}$ for an arbitrary group $\m{G}$.

\begin{Lemma}\label{l:derivablerules} 
For any $S \cup T \cup \{a, b\} \subseteq G$:
\begin{newlist}
\item[\rm (a)]  $S \cup \{\e\} \in \subrg{\m{G}}$;
\item[\rm (b)]  if $S \in \subrg{\m{G}}$, then $S \cup T \in \subrg{\m{G}}$;
\item[\rm (c)] if $S \in \subrg{\m{G}}$, then $S' \in \subrg{\m{G}}$ \,for some finite $S' \subseteq S$;
\item[\rm (d)] if $S \cup \{ab\} \in \subrg{\m{G}}$, then $S \cup \{a,b\} \in \subrg{\m{G}}$;
\item[\rm (e)] $S \in \subrg{\m{G}}$ if and only if $\sgr{S} \in \subrg{\m{G}}$.
\end{newlist}
\end{Lemma}
\begin{proof}
For (a), note that $\{\e,\iv{\e}\} = \{\e\} \subseteq S \cup \{\e\}$ and hence, by definition, $S \cup \{\e\} \in \subrg{\m{G}}$. The claims in (b) and (c) follow by a straightforward induction on  $k \in \N$ such that $S \in \subrgen{\m{G}}{k}$. For (d), observe that if $S \cup \{ab\} \in \subrg{\m{G}}$, then, by (b), $S \cup \{a,ab\} \in \subrg{\m{G}}$. But also, by definition,  $S \cup \{a,\iv{a}\} \in \subrg{\m{G}}$, and hence, since $b = \iv{a}ab$, we obtain $S \cup \{a,b\} \in \subrg{\m{G}}$. Finally, for (e), the left-to-right-direction follows directly from (b), and the right-to-left-direction follows by applying (b) and (c) to obtain a finite $S' \subseteq \sgr{S}$ such that $S \cup S' \in \subrg{\m{G}}$ and then applying (d) repeatedly to obtain $S \in \subrg{\m{G}}$. 
\end{proof}

\noindent
We now prove the left-to-right direction of Theorem~\ref{t:mainrightorderthm} part~(b), noting that for this direction there is no need to assume the right orderability of $\m{G}$.

\begin{Lemma}\label{l:rsound}
If $S \in \subrg{\m{G}}$, then $S$ does not extend to a right order of $\m{G}$.
\end{Lemma}
\begin{proof}
Using $(\dagger)$, it suffices to prove that for any $k \in \N$ and $S \in \subrgen{\m{G}}{k}$, there exist $c_1,\ldots,c_n \in G \setc \{\e\}$ such that for all $\de_1,\ldots,\de_n \in \{-1,1\}$,
					\[
					\e \in \sgr{S \cup \{c_1^{\de_1},\ldots,c_n^{\de_n}\}}.
					\]
We prove this claim by induction on $k$. For the base case, if  $S \in \subrgen{\m{G}}{0}$, then $\{a,\iv{a}\} \subseteq S$ for some $a \in G$, so $\e = a\iv{a} \in  \sgr{S}$. For the inductive step, suppose that $S = T \cup \{ab\} \in \subrgen{\m{G}}{k+1}$ because $T \cup \{a\} \in \subrgen{\m{G}}{k}$ and $T \cup \{b\} \in \subrgen{\m{G}}{k}$. By the induction hypothesis twice, we may assume without loss of generality that there exist $c_1,\ldots,c_n \in G \setc\{\e\}$ such that for all $\de_1,\ldots,\de_n \in \{-1,1\}$,
					\[
					\e \in \sgr{T \cup \{a,c_1^{\de_1},\ldots,c_n^{\de_n}\}}
					\quad \mbox{and} \quad
					\e \in \sgr{T \cup \{b,c_1^{\de_1},\ldots,c_n^{\de_n}\}}.
					\]
But then for all $\de_1,\ldots,\de_n,\de_{n+1} \in \{-1,1\}$, we obtain as required
					\[
					\e \in \sgr{T \cup \{ab,c_1^{\de_1},\ldots,c_n^{\de_m},a^{\de_{n+1}}\}}. \qedhere
					\]
\end{proof}


\noindent
To establish part~(a) and the right-to-left direction of part~(b) of Theorem~\ref{t:mainrightorderthm}, we prove two preparatory lemmas, related to Theorems~2.2 and~2.3 of~\cite{Con59}. 

\begin{Lemma}\label{l:starro}
For any $S \subseteq G$ such that $S \not \in \subrg{\m{G}}$,  there exists a subsemigroup $T$ of $\m{G}$ extending $S$ such that $T \not \in \subrg{\m{G}}$, and $G \setc T$ is a subsemigroup  of $\m{G}$.
\end{Lemma}
\begin{proof}
Suppose that $S \not \in \subrg{\m{G}}$ and consider the set\, $\mathcal{U}$ of all subsemigroups of $\m{G}$ extending $S$ that are not contained in $\subrg{\m{G}}$, partially ordered by inclusion. Clearly $\sgr{S} \in \mathcal{U}$ by Lemma~\ref{l:derivablerules}~(e). Moreover, if $(T_i)_{i \in I}$ is a chain in\, $\mathcal{U}$, then also $\bigcup_{i \in I} T_i \in \mathcal{U}$; otherwise $\bigcup_{i \in I} T_i \in \subrg{\m{G}}$ and, using Lemma~\ref{l:derivablerules}~(b) and (c), we would have $T_i \in \subrg{\m{G}}$ for some $i \in I$, a contradiction. Hence an application of Zorn's lemma yields a maximal element $T$ of\, $\mathcal{U}$.

To prove that $G \setc T$ is a subsemigroup of $\m{G}$, suppose that $b, c \in G \setc T$. Then $T \cup \{b\}$ and $T \cup \{c\}$ both properly extend $T$. Since $T$ is maximal, it follows that $T \cup \{b\}, T \cup \{c\} \in \subrg{\m{G}}$. Hence also $T \cup \{bc\} \in \subrg{\m{G}}$ and, since $T \not\in  \subrg{\m{G}}$, we obtain $bc \in  G \setc T$.
\end{proof}

\begin{Lemma}\label{l:rextensionlemma}
If $S \subseteq G$ satisfies $S \not \in \subrg{\m{G}}$ and $\{a\} \not \in \subrg{\m{G}}$ for all $a \in G \setc \{\e\}$, then $S$ extends to a right order of $\m{G}$.
\end{Lemma}
\begin{proof}
Suppose that $S \not \in \subrg{\m{G}}$ and $\{a\} \not \in \subrg{\m{G}}$ for all $a \in G \setc \{\e\}$. Consider the set $\mathcal{W}$ of all subsemigroups $T$ of $\m{G}$ extending $S$ such that $\e \not \in T$ and $G \setc T$ is a subsemigroup  of $\m{G}$, partially ordered by inclusion. It follows from Lemma~\ref{l:starro} that $\mathcal{W}$ is non-empty. Moreover, if $(T_i)_{i \in I}$ is a chain in $\mathcal{W}$, then also $\bigcup_{i \in I} T_i \in \mathcal{W}$. Hence an application of Zorn's lemma yields a maximal element $P$ of\, $\mathcal{W}$.

We claim that $P$ is a right order of $\m{G}$ extending $S$. Suppose for a contradiction that there exists  $a \in G \setc \{\e\}$ such that $a, \iv{a} \not \in P$. By Lemma~\ref{l:starro}, the assumption $\{a\}\not\in\subrg{\m{G}}$ yields a subsemigroup $T_a$ of $\m{G}$ containing $a$ such that $T_a \not \in \subrg{\m{G}}$ and $G \setc T_a$ is a subsemigroup of $\m{G}$. In particular, $\e \not \in T_a$. We claim that the maximality of $P$ is contradicted by
\[
P^{*} = P \cup \{b \in T_a\mid b,\iv{b} \not \in P\}.
\]
Observe first that $P^{*}$ properly extends $P$ and does not contain $\e$. It remains to prove that $P^{*}$ and $G \setc P^{*}$ are subsemigroups of $\m{G}$. 

Consider first $b,c \in P^*$. If $b,c \in P$, then $bc \in P \subseteq P^*$. Also, if $b, c \in T_a$ and $b,\iv{b}, c, \iv{c} \not \in P$, then $bc \in T_a$ ($T_a$ is a subsemigroup) and $bc, \iv{c}\iv{b} \not \in P$ ($G \setc P$ is a subsemigroup), so $bc \in P^*$. Suppose now  that $c \in P$ and $b \in T_a$ is such that $b, \iv{b} \not \in P$. Observe that $\iv{b}bc = c \in P$. Since $G \setc P$ is a subsemigroup and  $\iv{b} \not \in P$, we must have $bc \in P \subseteq P^*$. 

Now consider $b, c \not \in P^{*}$. In particular, $b,c \not \in P$, so $bc \not \in P$. There are three cases. If $b, c \not \in T_a$, then $bc \not \in T_a$ and hence, $bc \not \in P^{*}$. If $b, c \in T_a$, then, since $b, c \not \in P^{*}$, we must have $\iv{b},\iv{c} \in P$. So also $\iv{c}\iv{b} \in P$, and it follows that $bc \not \in P^{*}$. Suppose finally, without loss of generality, that $b \in T_a$ and $c \not \in T_a$. Since $b \in T_a$, $b \not \in P$, and $b \not \in P^{*}$, we must have $\iv{b} \in P$. Equivalently, $c\iv{c}\iv{b} \in P$, which, together with the fact that $c \not \in P$, implies $\iv{c}\iv{b} \in P$. Hence $bc \not \in P^{*}$.
\end{proof}

\begin{proof}[Proof of Theorem~\ref{t:mainrightorderthm}]
For (a), note first that if $\{a\} \not\in \subrg{\m{G}}$ for all $a \in G \setc \{\e\}$, then an application of Lemma~\ref{l:rextensionlemma} with $S = \emptyset$ yields a right order of $\m{G}$. For the converse direction, suppose that $\m{G}$ is right-orderable and let $a \in G \setc \{\e\}$. Then $\{a\}$ extends to a right order of $\m{G}$ and Lemma~\ref{l:rsound} yields $\{a\} \not\in \subrg{\m{G}}$.  Part~(b) now follows immediately from Lemma~\ref{l:rsound}~(a) and Lemma~\ref{l:rextensionlemma}.
\end{proof}


\section{Right orders on free groups and the word problem for free $\ell$-groups}\label{s:3}

In this section, we establish a correspondence between valid $\ell$-group equations and subsets of free groups that extend to right orders. We use this  correspondence to obtain new proofs of the decidability of the word problem for free $\ell$-groups (first proved by Holland and McCleary~\cite{HM79}) and the equivalent problem of checking when a given finite subset of a finitely generated free group extends to a right order (first proved by Clay and Smith~\cite{CS09}). We also obtain a new proof of the generation of $\vty{LG}$ as a variety by $\m{Aut}(\langle \R,\le \rangle)$ (first proved by Holland~\cite{Hol76}). 

Let $\m{T}(X)$ and $\m{T}^\ell(X)$ denote the term algebras  over a set $X$ of generators for the languages of groups and $\ell$-groups, respectively. Let $\m{F}(X)$ denote the free group over $X$, writing $\m{F}(k)$ when $|X| = k \in \mathbb{N}$. We write $\fr{t}$ for both $t \in T(X)$ and the corresponding member of $\m{F}(X)$ viewed as a reduced group term.

\begin{Theorem}\label{t:secondmainright}
The following are equivalent for $t_1,\ldots,t_n \in T(X)$:

\begin{newlist}

\item[\rm (1)]		$\{\fr{t}_1,\ldots,\fr{t}_n\}$ does not extend to a right order of $\m{F}(X)$;
			
\item[\rm (2)]		$\{\fr{t}_1,\ldots,\fr{t}_n\}  \in \subrg{\m{F}(X)}$;

\item[\rm (3)]		$\vty{LG} \models \e \le t_1 \lor \dots \lor t_n$.

\end{newlist}
\end{Theorem}
\begin{proof}
The equivalence (1) $\Leftrightarrow$ (2) follows immediately from Theorem~\ref{t:mainrightorderthm} and the fact that $\m{F}(X)$ is right-orderable (see, e.g.,~\cite{CR16}). 

We prove (2) $\Rightarrow$ (3) by induction on $k \in \N$ such that $\{\fr{t}_1,\ldots,\fr{t}_n\} \in \subrgen{\m{F}(X)}{k}$. For the base case, suppose that $\fr{t}_i = \iv{\fr{t}_j}$ in $\m{F}(X)$ for some $i,j \in \{1,\ldots,n\}$. Then $\vty{LG} \models t_i \eq \iv{t_j}$, so also $\vty{LG} \models \e \le t_1 \lor \dots \lor t_n$. For the inductive step, suppose that $t_n = uv$ and $\{\fr{t}_1,\ldots,\fr{t}_{n-1},\fr{uv}\} \in \subrgen{\m{F}(X)}{k+1}$ because
\[
\{\fr{t}_1,\ldots,\fr{t}_{n-1},\fr{u}\} \in \subrgen{\m{F}(X)}{k}\quad \text{and} \quad \{\fr{t}_1,\ldots,\fr{t}_{n-1},\fr{v}\} \in \subrgen{\m{F}(X)}{k}.
\]
By the induction hypothesis twice, 
\[
\vty{LG} \models \e \le t_1 \lor \dots \lor t_{n-1} \lor u \quad \text{and} \quad \vty{LG} \models \e \le t_1 \lor \dots \lor t_{n-1} \lor v.
\] 
Using the validity of the quasi-equation $(\e \le x\lor y)\,\&\, (\e \le x\lor z)\Rightarrow (\e \le x\lor yz)$ in $\vty{LG}$ (see, e.g.,~\cite[Lemma 3.3~(iv)]{GM16}), we obtain $\vty{LG} \models \e \le t_1 \lor \dots \lor t_{n-1} \lor uv$.

We prove (3) $\Rightarrow$ (1)  by contraposition. Suppose that $\{\fr{t}_1, \ldots, \fr{t}_n\}$ extends to a right order $\le$ of $\m{F}(X)$. Then $\fr{t}_1,\ldots,\fr{t}_n$ are all negative with respect to the dual right order $\le^\partial$  of $\m{F}(X)$.  Let $\varphi$ be the homomorphism from $\m{T}^\ell(X)$ to the $\ell$-group $\m{Aut}(\langle F(X), \le^\partial \rangle)$ with coordinate-wise lattice-ordering $\le^p$, defined by mapping each $x \in X$ to the order-automorphism $\varphi(x) \colon s \mapsto sx$. Then each $t \in T(x)$ is mapped to the order-automorphism $\varphi(t) \colon s \mapsto st$. In particular, $\varphi(t_i)(\e) = t_i <^\partial \e$ for  each $i \in \{1,\ldots,n\}$, and hence, since $<^\partial$ is total, for some $j \in \{1, \ldots, n\}$,
\[
\varphi(t_1 \lor \dots \lor t_n)(\e) = t_j <^\partial \e.
\] 
But $\varphi(\e)(\e) = \e$, so in $\m{Aut}(\langle F(X), \le^\partial \rangle)$,
\[
\varphi(\e) \not\le^p \varphi(t_1 \lor \dots \lor t_n).
\] 
Hence $\vty{LG} \not \models \e \le t_1 \lor \dots \lor t_n$. 
\end{proof}
 

Since, by $(\dagger)$, a set of elements of a group extends to a right order if and only if each of its finite subsets extends to a right order, we obtain the following result.

\begin{Corollary}
$S \subseteq T(X)$ extends to a right order of $\m{F}(X)$ if and only if $\vty{LG} \not \models \e \le t_1 \lor \dots \lor t_n$\, for all $\{t_1,\ldots,t_n\} \subseteq S$.
\end{Corollary}


Theorem~\ref{t:secondmainright} can be generalized to arbitrary right-orderable groups. Given a class $\vty{L}$ of $\ell$-groups and $\Si \cup \{s \eq t\} \subseteq (T^\ell(X))^2$, let $\Si \models_{\vty{L}} s \eq t$ denote that for any $\m{L} \in \vty{L}$ and homomorphism $\varphi \colon \m{T}^\ell(X) \to \m{L}$, whenever $\varphi(s') = \varphi(t')$ for all $s' \eq t' \in \Sigma$, also $\varphi(s) = \varphi(t)$. Recall also that a group presentation $\langle X \mid R \rangle$ identifies the quotient of the free group $\m{F}(X)$ by the normal subgroup generated by $R \subseteq T(X)$; for $t \in T(X)$, we let $\qr{t}$ denote the equivalence class of $t \in F(X)$ in this quotient.

\begin{Theorem}\label{t:generalmainright}
Suppose that $\m{G} = \langle X \mid R \rangle$ is a right-orderable group. Then the following are equivalent for $t_1,\ldots,t_n \in T(X)$:

\begin{newlist}

\item[\rm (1)]		$\{\qr{t}_1,\ldots,\qr{t}_n\}$ does not extend to a right order of $\m{G}$;
			
\item[\rm (2)]		$\{\qr{t}_1,\ldots,\qr{t}_n\}  \in \subrg{\m{G}}$;

\item[\rm (3)]		$\{r \eq \e\mid r \in R\} \models_{\vty{LG}} \e \le t_1 \lor \dots \lor t_n$.

\end{newlist}
\end{Theorem}
\begin{proof}
A slight modification of the proof of Theorem~\ref{t:secondmainright}.
\end{proof}

\begin{Example}\label{e:klein}
Consider the fundamental group $\m{K} = \langle x,y \mid xy\iv{x}y \rangle$  of the Klein bottle. It is easily shown that $\{xy\iv{x}y \eq \e\} \models_{\vty{LG}} \e \le \iv{y}\iv{x} \lor x$ and hence, by the preceding theorem, $\{\qr{\iv{y}\iv{x}},\qr{x}\}$ does not extend to a right order of $\m{K}$.
\end{Example}


We devote the rest of this section to decidability and generation problems. First, we recall the following result established by Holland and McCleary in~\cite{HM79}.

\begin{Theorem}\label{t:wordproblem}
The word problem for free $\ell$-groups is decidable.
\end{Theorem}

\noindent
The following decidability result is then an immediate consequence of Theorem~\ref{t:secondmainright}.

\begin{Theorem}\label{t:decidingextending}
The problem of deciding if a finite subset of a free group extends to a right order is decidable.
\end{Theorem}

\noindent
Galatos and Metcalfe have proved that the word problem for free $\ell$-groups is coNP-complete~\cite{GM16}, and it follows  that the problem of deciding whether or not a finite subset of a free group extends to a right order is also in the complexity class coNP. It is not known, however, if this latter problem is coNP-complete.

We now use the results of the previous section to present a proof of Theorem~\ref{t:wordproblem} that does not appeal to further algebraic results such as the Holland embedding theorem. As a byproduct, we obtain also an alternative proof that $\vty{LG}$ is generated as a variety by $\m{Aut}(\langle \R,\le \rangle)$. 

Given $S \subseteq T(X)$, let ${\rm is}(S)$ be the set of initial subterms (including the ``empty'' subterm $\e$) of elements of $S$ viewed as reduced terms, and let ${\rm cis}(S)$ be the set of all elements $u\iv{v}$ for distinct $u,v \in {\rm is}(S)$. 

\begin{Proposition}
The following are equivalent for any $t_1,\ldots,t_n \in T(X)$:

\begin{newlist}

\item[\rm (1)]	$\vty{LG} \models \e \le t_1 \lor \dots \lor t_n$;
					
\item[\rm (2)]	$\m{Aut}(\langle \R,\le \rangle) \models \e \le t_1 \lor \dots \lor t_n$;

\item[\rm (3)]	For $\{\fr{s}_1,\ldots,\fr{s}_m\}  = {\rm cis}(\{\fr{t}_1,\dots,\fr{t}_n\})$ and all $\de_1,\ldots,\de_m \in \{-1,1\}$, 
				\[
				\e \in \sgr{\{\fr{t}_1,\ldots,\fr{t}_n,\fr{s}_1^{\de_1},\ldots,\fr{s}_m^{\de_m}\}}.
				\]	
\end{newlist}
\end{Proposition}
\begin{proof}
(1) $\Rightarrow$ (2) is immediate. For (3) $\Rightarrow$ (1), observe that (3) implies, using $(\dagger)$, that $\{t_1,\ldots,t_n\}$ does not extend to a right order on $\m{F}(X)$, and hence, by Theorem~\ref{t:secondmainright}, that $\vty{LG} \models \e \le t_1 \lor \dots \lor t_n$.

(2) $\Rightarrow$ (3) is proved in some detail in~\cite{GM16}; we recall the main ingredients of this proof here for the sake of  completeness. Let $S = \{\fr{t}_1,\ldots,\fr{t}_n,\fr{s}_1^{\de_1},\ldots,\fr{s}_m^{\de_m}\}$ and suppose contrapositively that $\e \not\in \sgr{S}$ for some choice of $\de_1,\ldots,\de_m \in \{-1,1\}$. Let $a_{\fr{u}}$ be a variable for each $\fr{u} \in {\rm is}(\{\fr{t}_1,\dots,\fr{t}_n\})$. We define a set of inequations $T$ consisting of all $a_{\fr{u}} < a_{\fr{v}}$ such that $\fr{u},\fr{v} \in {\rm is}(\{\fr{t}_1,\dots,\fr{t}_n\})$ and $\fr{u}\iv{\fr{v}} \in S$. \\[.1in]
{\bf Claim 1.} $T$ is consistent over $\R$.\\[.1in]
{\em Proof.} Suppose for a contradiction that the set of inequations $T$ is inconsistent over $\R$. Then there exists a chain $a_{u_1} < a_{u_2} < \ldots < a_{u_k} < a_{u_1}$ in $T$. But then $\fr{u}_1\iv{\fr{u}_2},\fr{u}_2\iv{\fr{u}_3},\ldots,\fr{u}_{k}\iv{\fr{u}_1} \in S$ and it follows that $\e \in \sgr{S}$, contradicting our assumption. \qed\\[.1in]
By Claim~1, there exists a map sending each $a_{\fr{u}}$ for $\fr{u} \in {\rm is}(\{\fr{t}_1,\dots,\fr{t}_n\})$ to a real number $ r_{\fr{u}}$ that satisfies $T$. In particular, $\fr{t}_i\iv{\fr{\e}} \in S$ for each $i \in \{1,\ldots,n\}$, and hence $a_{\fr{t}_i} < a_{\fr{\e}} \in T$, yielding $r_{\fr{t}_i} < r_{\fr{e}}$. Now for each $x \in X$, define a partial map $\hat{x}$ from $\R$ to $\R$, that sends $r_{\fr{u}}$ to $r_{\fr{ux}}$ if $\fr{u},\fr{ux} \in {\rm is}(\{\fr{t}_1,\dots,\fr{t}_n\})$, and $r_{\fr{u\iv{x}}}$ to $r_{\fr{u}}$ if $\fr{u\iv{x}},\fr{u} \in {\rm is}(\{\fr{t}_1,\dots,\fr{t}_n\})$. \\[.1in]
{\bf Claim 2.} $\hat{x}$ is order-preserving.\\[.1in]
{\em Proof.}  Suppose first that $\hat{x}$ maps  $r_{\fr{u}}$ to $r_{\fr{ux}}$ and  $r_{\fr{v}}$ to $r_{\fr{vx}}$, but $r_{\fr{u}} < r_{\fr{v}}$ and $r_{\fr{vx}} < r_{\fr{ux}}$. Then $\fr{u}\iv{\fr{v}} \in S$ and $(\fr{vx})(\iv{\fr{x}}\iv{\fr{u}}) = \fr{v}\iv{\fr{u}} \in S$, so $\e \in \sgr{S}$, a contradiction. Alternatively, suppose that $\hat{x}$ maps  $r_{\fr{u}}$ to $r_{\fr{ux}}$ and  $r_{\fr{v\iv{x}}}$ to $r_{\fr{v}}$, but $r_{\fr{u}} < r_{\fr{v\iv{x}}}$ and $r_{\fr{v}} < r_{\fr{ux}}$. Then $\fr{u}\iv{(\fr{v\iv{x}})} = \fr{ux\iv{v}} \in S$ and $\fr{v}\iv{(\fr{ux})} = \fr{v\iv{x}\iv{u}} \in S$, so $\e \in \sgr{S}$, a contradiction. Other cases are very similar. \qed\\[.1in]
 Finally, we extend each $\hat{x}$ for $x \in X$ linearly to a function $\hat{\varphi}(x)$ in $\m{Aut}(\langle \R,\le \rangle)$. But then $\hat{\varphi}(t_i)(r_{\fr{e}}) = r_{\fr{t}_i}$ for each $i \in \{1,\ldots,n\}$, while  $\hat{\varphi}(\e)(r_{\fr{e}}) = r_{\fr{e}}$. Hence, $\hat{\varphi}(t_1 \lor \dots \lor t_n)(r_{\fr{e}}) = r_{\fr{t}_j} < r_{\fr{e}}$ for some $j \in \{1, \ldots, n\}$. This establishes $\m{Aut}(\langle \R,\le \rangle) \not\models \e \le t_1 \lor \dots \lor t_n$ as required.
\end{proof}

\noindent
Implicit in the proof of this proposition is a decision procedure for checking the validity of an inequation $\e \le t_1 \lor \dots \lor t_n$ in $\vty{LG}$, and hence a proof of Theorem~\ref{t:wordproblem}. Namely, calculate ${\rm cis}(\{\fr{t}_1,\dots,\fr{t}_n\})$ and denote its elements by $s_1,\ldots,s_m$. For each choice of $\de_1,\ldots,\de_m \in \{-1,1\}$ and $S = \{\fr{t}_1,\ldots,\fr{t}_n,\fr{s}_1^{\de_1},\ldots,\fr{s}_m^{\de_m}\}$, check the satisfiability of the resulting set of inequations $T$ over $\R$. If all of these sets are unsatisfiable, then we have established the equivalent condition (3) above; otherwise, $\vty{LG} \not \models \e \le t_1 \lor \dots \lor t_n$. 

The equivalence of (1) and (2) also yields the following result, first proved by Holland~\cite{Hol76}.

\begin{Corollary}
The variety $\vty{LG}$ of $\ell$-groups is generated by $\m{Aut}(\langle \R,\le \rangle)$.
\end{Corollary}

\noindent
This generation result can be interpreted in terms of extending subsets of free groups to right orders.

\begin{Corollary}
A set of elements $\{t_1,\ldots,t_n\}$ of a free group over $x_1,\ldots,x_k$ extends to a right order if and only if there exist order-preserving bijections $\hat{x}_1,\ldots,\hat{x}_k$ of the real number line  such that $\hat{t_i}(0) < 0$ for all $i \in \{1, \ldots, n\}$. 
\end{Corollary}

As mentioned above, Theorem~\ref{t:decidingextending} is a direct consequence of Theorems~\ref{t:secondmainright} and~\ref{t:wordproblem}. However, an algorithm for recognizing when a finite subset of a finitely generated free algebra $\m{F}(k)$ extends to a right order was already provided by Clay and Smith in~\cite{CS09}.  Let $|t|$ denote the {\em length} of a reduced term $t$ in $\m{F}(k)$, and for $l \in \N$, let $F_l(k)$ denote the set of all elements of $\m{F}(k)$ of length at most $l$. Note that $F_{l}(k)$ is finite, and can be viewed as the $l$-ball of the Cayley graph of $F_{l}(k)$ relative to $X$. For a subset $S$ of $\m{F}(k)$ which omits $\e$, we say that $S$ is an {\em $l$-truncated right order} on $\m{F}(k)$ if $S = \sgr{S} \cap F_l(k)$, and for all $t \in F_{l-1}(k) \setc \{\e\}$, either $t \in S$ or $\iv{t} \in S$. 

\begin{Proposition}[Clay and Smith~\cite{CS09}]\label{p:claysmith}
 $S \subseteq F(k)$ extends to a right order of $\m{F}(k)$ if and only if $S$ extends to an $l$-truncated right order of $\m{F}(k)$ for some $l \in \N$.
\end{Proposition}

\noindent 
The condition described above can be decided for finite $S$ as follows. Let $l$ be the maximal length of an element in $S$. Extend $S$ to the finite set $S^{*}$ obtained by adding $st$ whenever $s,t$ occur in the set constructed so far and $|st| \le l$. This ensures that $S^{*} = \sgr{S^{*}} \cap F_l(k)$. If $\e \in S^{*}$, then stop. Otherwise, given $t\in F_{l-1}(k) \setc \{\e\}$ such that $t \not \in S^{*}$ and $\iv{t} \not \in S^*$, add $t$ to $S^{*}$ to obtain $S_1$ and $\iv{t}$ to $S^{*}$ to obtain $S_2$, and repeat the process with these sets. This procedure eventually terminates because $F_l(k)$ is finite. Hence, this algorithm can be used to decide whether a finite subset of a finitely generated free group extends to a right order, resulting in a proof of Theorem~\ref{t:decidingextending}, and therefore also of Theorem~\ref{t:wordproblem}.

\begin{Example}
Consider $S = \{xx,yy,\iv{x}\iv{y}\} \subseteq F(2)$. By adding all products in $F_2(2)$ of members of $S$, we obtain
\[
S^{*}= \{xx,yy,\iv{x}\iv{y},x\iv{y},\iv{x}y,xy\}.
\]
We then consider all possible signs $\de$ for $x,y \in F_1(2)$. If we add $\iv{x}$ or $\iv{y}$ to $S^{*}$ and take products, then clearly, using $xx$ or $yy$, we obtain $\e$. Similarly, if we add $x$ and $y$ to $S^{*}$, then, taking products, using $\iv{x}\iv{y}$, we obtain $\e$. Hence we may conclude that $S$ does not extend to a right order of $\m{F}(2)$ and obtain
\[
\vty{LG} \models \e \le xx \lor yy \lor \iv{x}\,\iv{y}.
\]
Consider now $T = \{xx, xy,y\iv{x}\} \subseteq F(2)$. By adding all products in $F_2(2)$ of members of $T$, we obtain
\[
T^{*}= \{xx, xy,y\iv{x},yx,yy\}.
\]
We choose $x,y \in F_1(2)$ to be positive and obtain $\{xx, xy,y\iv{x},yx,yy,x,y\}$, a $2$-truncated right order of $\m{F}(2)$. Hence $T$ extends to a right order of $\m{F}(2)$ and
\[
\vty{LG} \not \models \e \le xx \lor xy \lor y\iv{x}.
\]
\end{Example}

We conclude this section by mentioning a topological result regarding right orders of free groups, and its interpretation in terms of validity in $\ell$-groups.  Note first that a right order $P$ is isolated in the space $\mathcal{RO}(\m{G})$ of right orders of a group $\m{G}$ mentioned in the introduction  if and only if it is the unique right order extending some finite subset of $G$. In~\cite{mccleary85} McCleary proved the following for the spaces of right orders of non-trivial finitely generated free groups.

\begin{Theorem}[McCleary~\cite{mccleary85}]\label{t:isolated}
The space of right orders of a non-trivial finitely generated free group has no isolated points.
\end{Theorem}

\noindent
By Theorem~\ref{t:secondmainright}, we obtain the following feature of validity in $\ell$-groups.

\begin{Corollary}
The following are equivalent for $t_1,\ldots,t_n \in T(k)$:

\begin{newlist}

\item[\rm (1)]	$\vty{LG} \models \e \le t_1 \lor \dots \lor t_n$;

\item[\rm (2)] $\vty{LG} \models \e \le t_1 \lor \dots \lor t_n\lor s\,\text{ or }\,
\vty{LG} \models \e \le t_1 \lor \dots \lor t_n\lor \iv{s}$\, for all $s \in T(k)$.

\end{newlist}
\end{Corollary}

\begin{proof}
(1) $\Rightarrow$ (2) is immediate.  For (2) $\Rightarrow$ (1),  suppose that $\vty{LG} \not\models \e \le t_1 \lor \dots \lor \,t_n$. By Theorem~\ref{t:secondmainright}, the set $\{t_1, \ldots, t_n\}$ extends to a right order of $\m{F}(k)$. But, by Theorem~\ref{t:isolated}, the space $\mathcal{RO}(\m{F}(k))$ has no isolated points, so there exists $s \in T(k)$ such that both $\{t_1, \ldots, t_n, s\}$ and $\{t_1, \ldots, t_n, \iv{s}\}$ extend to right orders of $\m{F}(k)$. Theorem~\ref{t:secondmainright} then yields $\vty{LG} \not\models \e \le t_1 \lor \dots \lor t_n\lor s$ and $\vty{LG} \not\models \e \le t_1 \lor \dots \lor t_n\lor \iv{s}$ as required.
\end{proof}

\noindent
Let us remark finally that Proposition~\ref{p:claysmith} can be used to show that every isolated point of $\mathcal{RO}(\m{F}(k))$  ($k \ge 2$) is finitely generated as a semigroup, which, together with a result of Kielak~\cite{kie15} that {\em no} right order of $\m{F}(k)$ is finitely generated as a semigroup, yields an alternative proof of Theorem~\ref{t:isolated} (see~\cite{deroin2014groups} for details). 


\section{Ordering relatively free groups and validity in ordered groups}\label{s:4}

We turn our attention in this section to orders on (relatively free) groups and validity of equations in corresponding classes of ordered groups (equivalently, varieties of representable $\ell$-groups). We begin by providing a characterization of subsets of a group that extend to orders. Since the results and proofs are very similar to those presented for right orders in Section~\ref{s:2}, we confine ourselves here to presenting the main ingredients of the approach,  pointing out only the most significant differences. 

Let us again fix a group $\m{G}$ with neutral element $\e$ and denote by $\nsgr{S}$ the normal subsemigroup of $\m{G}$ generated by $S\subseteq G$. Clearly, $\nsgr{S}$ is a partial order of $\m{G}$ if and only if $\e \not \in \nsgr{S}$. As in the case of right orders, the following characterization of subsets of $G$ that extend to orders of $\m{G}$ is established by a straightforward application of Zorn's lemma (see~\cite{Fuc63}):

\begin{newlist}
\item[$(\ddagger)$]
 $S \subseteq G$ extends to an order of $\m{G}$ if and only if for all $a_1,\ldots,a_n \in G \setc\{\e\}$, there exist $\de_1,\ldots,\de_n \in \{-1,1\}$ such that $\e \not \in \nsgr{S \cup \{a_1^{\de_1},\ldots,a_n^{\de_n}\}}$.
\end{newlist}

\noindent
Our alternative characterization (similar again to those obtained in~\cite{OM52} and~\cite{Con59}) supplements the characterization for right orders with an extra condition to take care of normality. We define inductively for $n \in \N$,
\[
\begin{array}{rcl}
\subgen{\m{G}}{0} & = & \{S \subseteq G \mid S \cap S^{-1} \neq \emptyset\};\\[.1in]
\subgen{\m{G}}{n+1} & = & \subgen{\m{G}}{n}  \cup \{T \cup \{ab\} \mid T \cup \{a\}, T \cup \{b\} \in \subgen{\m{G}}{n} \\[.05in]
& & \hspace*{3.75cm} \text{ or }\, T \cup \{ba\} \in \subgen{\m{G}}{n}\};\\[.1in]
\subg{\m{G}} & = & \bigcup_{n \in \N} \subgen{\m{G}}{n}.
\end{array}
\]
It follows that $\subg{\m{G}} \subseteq \mathcal{P}(G)$ is the smallest set containing $\subgen{\m{G}}{0}$ with the properties that if $T \cup \{a\}, T \cup \{b\}  \in \subg{\m{G}}$, then $T \cup \{ab\} \in \subg{\m{G}}$, and if $T \cup \{ba\}  \in \subg{\m{G}}$, then $T \cup \{ab\} \in \subg{\m{G}}$.

\begin{Example}
In the free group $\m{F}(2)$ generated by $x$ and $y$, clearly $\{x, \iv{x}\} \in \subgen{\m{F}(2)}{0}$, and hence $\{x, y\iv{x}\iv{y}\} \in \subgen{\m{F}(2)}{1}$. This corresponds to the fact that $\{x,y\iv{x}\iv{y}\}$ does not extend to an order on  $\m{F}(2)$ and also the fact that the inequation $\e \le x \lor y\iv{x}\iv{y}$ is valid in all o-groups (see Theorem~\ref{t:secondmain}). Note, however, that $\{x, y\iv{x}\iv{y}\} \not \in \subrg{\m{F}(2)}$, reflecting the fact that  $\{x,y\iv{x}\iv{y}\}$ does extend to a right order on  $\m{F}(2)$ and  the fact that $\e \le x \lor y\iv{x}\iv{y}$ is not  valid in all $\ell$-groups (see Theorem~\ref{t:secondmainright}).
\end{Example}

\begin{Example}
In the fundamental group $\m{K}$ of the Klein bottle (see Example~\ref{e:klein}), $\{y, xy\iv{x}\} \in  \subgen{\m{K}}{0}$ and hence $\{y\} \in  \subgen{\m{K}}{1}$. This corresponds to the fact that $\m{K}$ is not orderable (Theorem~\ref{t:mainorderthm}).
\end{Example}

\noindent
The proof of the following theorem proceeds similarly to the proof of Theorem~\ref{t:mainrightorderthm}, using condition $(\ddagger)$ to establish the left-to-right direction of (b), and  normal subsemigroups of $\m{G}$ extending $S \subseteq G$ to establish analogues of Lemmas~\ref{l:starro} and~\ref{l:rextensionlemma}.

\begin{Theorem}\label{t:mainorderthm} \ 
\begin{newlist}
\item[\rm (a)]
A group $\m{G}$ is orderable if and only if  $\{a\} \not\in \subg{\m{G}}$ for all $a \in G \setc \{\e\}$.
\item[\rm (b)]
If a group $\m{G}$ is orderable, then  $S \subseteq G$ extends to an order of $\m{G}$ if and only if  $S\not\in\subg{\m{G}}$.
\end{newlist}
\end{Theorem}


We now establish a correspondence between subsets of relatively free groups of a variety of groups that extend to orders, and the valid  $\ell$-group equations of a corresponding class of ordered groups. Let $\vty{V}$ be a variety of groups, and let $\vty{K}_\vty{V}$  be the class of all o-groups that have a group reduct in $\vty{V}$. In particular, if $\vty{V}$ is the variety of all groups, then $\vty{K}_\vty{V}$ is the class of all o-groups. We consider a relatively free group $\m{F}_\vty{V}(X)$ of $\vty{V}$ over some non-empty set $X$ of generators, writing again $\fr{t}$ to denote the reduced term in $\m{F}_\vty{V}(X)$ corresponding to $t \in T(X)$.

\begin{Theorem}\label{t:secondmain}
If $\m{F}_\vty{V}(X)$ is orderable, then the following are equivalent for all $t_1,\ldots,t_n \in T(X)$:

\begin{newlist}

\item[\rm (1)]		$\{\fr{t}_1,\ldots,\fr{t}_n\}$ does not extend to an order of $\m{F}_\vty{V}(X)$;

\item[\rm (2)]		$\{\fr{t}_1,\ldots,\fr{t}_n\}  \in \subg{\m{F}_\vty{V}(X)}$;

\item[\rm (3)]		$\vty{K}_\vty{V} \models \e \le t_1 \lor \dots \lor t_n$.

\end{newlist}
\end{Theorem}
\begin{proof}
The equivalence (1) $\Leftrightarrow$ (2) follows immediately from Theorem~\ref{t:mainorderthm}.

We prove (2) $\Rightarrow$ (3) by induction on $k \in \N$ such that  $\{\fr{t}_1,\ldots,\fr{t}_n\} \in \subgen{\m{F}_\vty{V}(X)}{k}$. The base case follows exactly as in the proof of Theorem~\ref{t:secondmainright}, as does the inductive step for the case where $\{\fr{t}_1,\ldots,\fr{t}_{n-1},\fr{uv}\} \in \subgen{\m{F}(X)}{k+1}$ results from $\{\fr{t}_1,\ldots,\fr{t}_{n-1},\fr{u}\} \in \subgen{\m{F}(X)}{k}$ and $\{\fr{t}_1,\ldots,\fr{t}_{n-1},\fr{v}\} \in \subgen{\m{F}(X)}{k}$. Suppose now that $\{\fr{t}_1,\ldots,\fr{t}_{n-1},\fr{u}\fr{v}\} \in \subgen{\m{F}_\vty{V}(X)}{k+1}$ because $\{\fr{t}_1,\ldots,\fr{t}_{n-1},\fr{v}\fr{u}\} \in \subgen{\m{F}_\vty{V}(X)}{k}$. By the induction hypothesis, 
\[
\vty{K}_\vty{V} \models \e \le t_1 \lor \dots \lor t_{n-1} \lor vu,
\] 
and hence, since the quasi-equation $(\e \le x\lor yz)\Rightarrow (\e \le x\lor zy)$ is valid in all o-groups,  $\vty{K}_\vty{V} \models \e \le t_1 \lor \dots \lor t_{n-1} \lor uv$. 

We prove (3) $\Rightarrow$ (1) by contraposition. Suppose that $\{\fr{t}_1,\ldots,\fr{t}_n\}$ extends to an order $\le$ of $\m{F}_\vty{V}(X)$. Then $\fr{t}_1,\ldots,\fr{t}_n$ are all negative with respect to the dual order $\le^\partial$ on $\m{F}_\vty{V}(X)$. Let $\m{L} \in \vty{K}_\vty{V}$ be the o-group with group reduct $\m{F}_\vty{V}(X)$ and order $\le^\partial$, and let $\varphi$ be the homomorphism from $\m{T}^\ell(X)$ to  $\m{L}$ defined by mapping each $x \in X$ to $\fr{x} \in F_\vty{V}(X)$. Then $\varphi(t)$ is $\fr{t} \in F_\vty{V}(X)$ for each group term $t$, and, since $\le^\partial$ is total, $\varphi(t_1 \lor \dots \lor t_n) = \varphi(t_i) = \fr{t}_i \in F_\vty{V}(X)$ for some $i \in \{1,\ldots,n\}$. 
But given that $\fr{t}_1,\ldots,\fr{t}_n$ are all negative, also 
\[
\varphi(t_1 \lor \dots \lor t_n) = \fr{t}_i <^\partial \fr{\e} = \varphi(\e).
\] 
So  $\vty{K}_\vty{V} \not \models \e \le t_1 \lor \dots \lor t_n$.
\end{proof}

\begin{Corollary}
If $\m{F}_\vty{V}(X)$ is orderable, then $ \fr{S} \subseteq F(X)$ extends to an order of $\m{F}_\vty{V}(X)$ if and only if $\vty{K}_\vty{V} \not \models \e \le t_1 \lor \dots \lor t_n\,$ for all $\{t_1,\ldots,t_n\} \subseteq S$.
\end{Corollary}

We may also view these results from the opposite direction, starting with some particular variety of $\ell$-groups rather than a variety of groups. Let $\vty{L}$ be a variety of  representable $\ell$-groups defined relative to $\vty{RG}$ by a set $\Si \subseteq (T(X))^{2}$ of group equations, and let $\vty{V}$ be the variety of groups defined by $\Si$. If the relatively free group $\m{F}_\vty{V}(X)$  is orderable, then Theorem~\ref{t:secondmain} implies that  $\{\fr{t}_1,\dots,\fr{t}_n\}\subseteq T(X)$ extends to an order of $\m{F}_{\vty{V}}(X)$ if and only if $\vty{L} \not\models \e \le t_1\lor\dots\lor t_n$. Note also that in this case, if an $\ell$-group equation fails in $\vty{L}$, then it fails in $\m{F}_\vty{V}(X)$ equipped with some order. So, if  the relatively free groups of $\vty{L}$ are orderable, then $\vty{L}$ is generated as a variety by the class of o-groups with group reducts of $\m{F}_\vty{V}(X)$. In particular, the variety of representable $\ell$-groups is generated as a variety by the class of o-groups whose group reducts are free groups. 


Finally, let us remark that some questions discussed in Section~\ref{s:3} for right orders and $\ell$-groups are still open for orders and representable $\ell$-groups. It follows from Theorem~\ref{t:secondmain} that the decidability of the word problem for free representable $\ell$-groups is equivalent to the problem of deciding whether a given finite set of elements of a free group extends to an order. Both of these problems are, as far as we know, still open. Also from a topological perspective, much less is known in the case of orders. The topological space $\mathcal{BO}(\m{G})$ of orders of a group $\m{G}$ consists of the set of normal subsemigroups $P$ such that $P \cup \iv{P} = G \setc \{\e\}$, equipped with the induced powerset topology, and forms a closed subspace of $\mathcal{RO}(\m{G})$. It is not known, however, whether the space $\mathcal{BO}(\m{F}(k))$ of orders of a (non-trivial) finitely generated free group has any isolated point. Equivalently, it is not known if for $t_1,\ldots,t_n \in T(k)$, whenever $\vty{RG} \models \e \le t_1 \lor \dots \lor t_n\lor s\,\text{ or }\, \vty{RG} \models \e \le t_1 \lor \dots \lor t_n\lor \iv{s}$\, for all $s \in T(k)$, then $\vty{RG} \models \e \le t_1 \lor \dots \lor t_n$.

\bibliographystyle{plain}

\end{document}